\documentclass[a4paper,10pt]{amsproc}
\usepackage{amsfonts, amsmath, amssymb, graphicx, mathrsfs, tikz, xcolor, hyperref}
\usepackage[a4paper, total={6in, 10in}]{geometry}
\hypersetup{
	colorlinks   = true, 
	urlcolor     = blue, 
	linkcolor    = green, 
	citecolor   = red 
}
\tikzset{Bullet/.style={fill=black,draw,color=#1,circle,minimum size=2pt,scale=0.45}}
\usetikzlibrary{shapes.geometric}
\usetikzlibrary{positioning, tikzmark}
\usepackage[all]{xy}
\theoremstyle{plain}

\newtheorem{theorem}{Theorem}[section]
\newtheorem{lemma}[theorem]{Lemma}

\theoremstyle{definition}
\newtheorem{definition}[theorem]{Definition}

\newtheorem{remark}[theorem]{Remark}

\newtheorem{counter example}[theorem]{Counter Example}

\newtheorem{corollary}[theorem]{Corollary}
\newtheorem{example}[theorem]{Example}

\numberwithin{equation}{section}

\begin{document}
	\Large{\title{More on generalizations of topology of uniform convergence and $m$-topology on $C(X)$}
	\author[P. Nandi]{Pratip Nandi}
	\address{Department of Pure Mathematics, University of Calcutta, 35, Ballygunge Circular Road, Kolkata 700019, West Bengal, India}
	\email{pratipnandi10@gmail.com}
	\author[R. Bharati]{Rakesh Bharati}
	\address{Department of Pure Mathematics, University of Calcutta, 35, Ballygunge Circular Road, Kolkata 700019, West Bengal, India}
	\email{rbharati.rs@gmail.com}
	\author[A. Deb Ray]{Atasi Deb Ray}
	\address{Department of Pure Mathematics, University of Calcutta, 35, Ballygunge Circular Road, Kolkata 700019, West Bengal, India}
	\email{debrayatasi@gmail.com}
	\author[S.K. Acharyya]{Sudip Kumar Acharyya}
	\address{Department of Pure Mathematics, University of Calcutta, 35, Ballygunge Circular Road, Kolkata 700019, West Bengal, India}
	\email{sdpacharyya@gmail.com}
		
	\thanks{The first author thanks the CSIR, New Delhi – 110001, India, for financial support.}
	\thanks{The second author acknowledges financial support from University Grants Commission, New Delhi, for the award of research fellowship (F. No. 16-9(June 2018)/2019(NET/CSIR))}
		
	\begin{abstract}
	This paper conglomerates our findings on the space $C(X)$ of all real valued continuous functions, under different generalizations of the topology of uniform convergence and the $m$-topology. The paper begins with answering all the questions which were left open in our previous paper on the classifications of $Z$-ideals of $C(X)$ induced by the $U_I$ and the $m_I$-topologies on $C(X)$. Motivated by the definition of $m^I$-topology, another generalization of the topology of uniform convergence, called $U^I$-topology, is introduced here. Among several other results, it is established that for a convex ideal $I$, a necessary and sufficient condition for $U^I$-topology to coincide with $m^I$-topology is the boundedness of $X\setminus\bigcap Z[I]$ in $X$. As opposed to the case of the $U_I$-topologies (and $m_I$-topologies), it is proved that each $U^I$-topology (respectively, $m^I$-topology) on $C(X)$ is uniquely determined by the ideal $I$. In the last section, the denseness of the set of units of $C(X)$ in $C_U(X)$ (= $C(X)$ with the topology of uniform convergence) is shown to be equivalent to the strong zero dimensionality of the space $X$. Also, the space $X$ is a weakly P-space if and only if the set of zero divisors (including 0) in $C(X)$ is closed in $C_U(X)$. Computing the closure of $C_\mathscr{P}(X)$ (=$\{f\in C(X):\text{the support of }f\in\mathscr{P}\}$ where $\mathscr{P}$ denotes the ideal of closed sets in $X$) in $C_U(X)$ and $C_m(X)$ (= $C(X)$ with the $m$-topology), the results $cl_UC_\mathscr{P}(X) = C_\infty^\mathscr{P}(X)$ ($=\{f\in C(X):\forall n\in\mathbb{N}, \{x\in X:|f(x)|\geq\frac{1}{n}\}\in\mathscr{P}\}$) and  $cl_mC_\mathscr{P}(X)=\{f\in C(X):f.g\in C^\mathscr{P}_\infty(X)\text{ for each }g\in C(X)\}$ are achieved.
	\end{abstract}

	\subjclass[2020]{Primary 54C40; Secondary 46E30}
	\keywords{Topology of uniform convergence, $m$-topology, $U_I$, $m_I$, $U^I$ and $m^I$ topologies,  weakly P-space, strongly 0-dimensional space, $C_\mathscr{P}(X)$, $C_\infty^\mathscr{P}(X)$.}                                    
	
	\maketitle
	
	\section{Introduction}
	In the entire article $X$ designates a completely regular Hausdorff space. As is well known $C(X)$ stands for the ring of real valued continuous functions on $X$. Suppose $C^*(X)=\{f\in C(X):f\text{ is bounded on }X\}$. If for $f\in C(X)$ and $\epsilon>0$ in $\mathbb{R}$, $U(f,\epsilon)=\{g\in C(X):\sup\limits_{x\in X}|f(x)-g(x)|<\epsilon\}$, then the family $\{U(f,\epsilon):f\in C(X),\epsilon>0\}$ turns out to be an open base for the so-called topology of uniform convergence or in brief the $U$-topology on $C(X)$. Several experts have studied $U$-topology on $C(X)$, from various points of view. One can look at the articles [\cite{Gillman},~\cite{McCoy},~\cite{Plank1}] for a glimpse of some relevant facts about this topology. A generalization of this $U$-topology on $C(X)$ via a kind of ideal in $C(X)$, viz a $Z$-ideal $I$ in $C(X)$, is already studied only recently~\cite{Bharati}. Incidentally the collection $\{U_I(f,\epsilon):f\in C(X),\epsilon>0\}$ constitutes an open base for this generalized $U$-topology, named as the $U_I$-topology on $C(X)$. Here $U_I(f,\epsilon)=\{g\in C(X):\text{there exists }Z\in Z[I]\equiv\{Z(h):h\in I\}\text{ such that }\sup\limits_{x\in Z}|f(x)-g(x)|<\epsilon\}$, $Z(h)$ standing for the zero set of the function $h$. It is worth mentioning in this context that an analogous type of topology, viz the $m_I$-topology, on $C(X)$ is introduced and investigated in some detail in~\cite{Azarpanah2}. Here $I$ is a $Z$-ideal in $C(X)$ and a typical basic open neighborhood of $f\in C(X)$ in this topology looks like: $m_I(f,u)=\{g\in C(X):|f(x)-g(x)|<u(x)\text{ for all }x\in Z\text{ for some }Z\in Z[I]\}$, here $u\in C(X)$ and is strictly positive on some $Z_0\in Z[I]$. With the special choice $I=(0)$, the $m_I$-topology and $U_I$-topology reduce respectively to the well-known $m$-topology and $U$-topology on $C(X)$ [2M, 2N~\cite{Gillman}]. In section $4$ of the article~\cite{Bharati}, two classifications of $Z$-ideals in $C(X)$ induced by $U_I$-topologies and also by the $m_I$-topologies are defined. To be more specific binary relations `$\sim$' and `$\approx$' on the set $\mathcal{I}$ of all $Z$-ideals in $C(X)$ are introduced as follows: for $I,J\in\mathcal{I},\ I\sim J$ if $U_I$-topology $=U_J$-topology and $I\approx J$ if $m_I$-topology $=m_J$-topology. For $I\in\mathcal{I}$, we set $[I]=\{J\in\mathcal{I}:U_I\text{-topology }=U_J \text{-topology}\}$ and $[[I]]=\{J\in\mathcal{I}:m_I\text{-topology }=m_J \text{-topology}\}$. It is established in~\cite{Bharati}, [Theorem $4.1$ and Theorem $4.4$] that each equivalence class $[I]$ has a largest member and analogously all the equivalence classes $[[I]]$ also have largest members [Theorem $4.13$ and Theorem $4.20$,~\cite{Bharati}]. It is further realized that some of these equivalence classes (in both these classifications of $Z$-ideals in $C(X)$) have smallest members too [Theorem $4.10$, Theorem $4.21$,~\cite{Bharati}].\\
	In section~\ref{Sec2} of the present article we prove that each equivalence class $[I]$ and $[[I]]$ has a smallest member, thereby answering the questions $4.26$ and $4.27$ asked in~\cite{Bharati} affirmatively. Again it was established in~\cite{Bharati} that if $X$ is a $P$-space, then each equivalence class $[I]$ and $[[I]]$ degenerates into singleton [Theorem $4.12$ and Theorem $4.23$ in~\cite{Bharati}] and hence $\sim$ and $\approx$ are identical equivalence relations on $\mathcal{I}$. In this article we check that, regardless of whether or not $X$ is a $P$-space, $\sim$ and $\approx$ are indeed identical equivalence relations on $\mathcal{I}$, the set of all $Z$-ideals on $C(X)$. This answers negatively the question $4.25$ asked in~\cite{Bharati}.\\
	In section~\ref{Sec3} of this article we introduce yet another generalization of $U$-topology on $C(X)$, this time via an ideal $I$ of $C(X)$ [not necessarily a $Z$-ideal nor even a proper ideal] but with a slightly different technique. Essentially for $f\in C(X)$ and $\epsilon>0$, we set $\widetilde{B}(f,I,\epsilon)=\{g\in C(X):\sup\limits_{x\in X}|f(x)-g(x)|<\epsilon\text{ and }f-g\in I\}$. Then it needs a few routine computation to show that the family $\{\widetilde{B}(f,I,\epsilon):f\in C(X), \epsilon>0\}$ makes an open base for some topology on $C(X)$, which we designate by the $U^I$-topology on $C(X)$. It is not at all hard to check that $C(X)$ with this $U^I$-topology is an additive topological group. The $U$-topology on $C(X)$ is a special case of the $U^I$-topology with $I=C(X)$. Let us mention at this point that an analogous kind of topology, viz the $m^I$-topology on $C(X)$, is initiated and studied in some details in~\cite{Azarpanah}. A typical basic open neighborhood of $f\in C(X)$ for this latter topology is a set of the form $\{g\in C(X):|f(x)-g(x)|<u(x)\text{ for all } x\in X\text{ and }f-g\in I\}$, here $u$ is a positive unit in $C(X)$. $C(X)$ with the $m^I$-topology is a topological ring as is proved in~\cite{Azarpanah}. For notational convenience, we let $C_{U^I}(X)$ to stand for $C(X)$ equipped with the $U^I$-topology. Analogously $C_{m^I}(X)$ designates $C(X)$ with the $m^I$-topology. In general the $U^I$-topology on $C(X)$ is weaker than the $m^I$-topology. Incidentally it is proved [vide Theorem~\ref{Th3.8}] that if $I$ is a convex ideal in $C(X)$ (in particular $I$ may be a $Z$-ideal in $C(X)$), then $U^I$-topology $=m^I$-topology if and only if $X\setminus\bigcap Z[I]$ is a bounded subset of $X$. We observe that $I\cap C^*(X)$ is a clopen set in the space $C_{U^I}(X)$ [Theorem~\ref{Th3.7}(\ref{Th3.7b})]. We use this fact to show that $I\cap C^*(X)$ is indeed the component of $0$ in $C_{U^I}(X)$ [Theorem~\ref{Th3.13}]. We recall that a topological space $Y$ is called homogeneous if given any two points $p,q\in Y$, there exists a homoeomorphism $\phi:Y\to Y$ such that $\phi(p)=q$. A topological group is a natural example of a homogeneous space. It follows that $C_{U^I}(X)$ is either locally compact or nowhere locally compact, indeed the latter criteria holds when and only when $X\setminus \bigcap Z[I]$ is a finite set [Theorem~\ref{Th3.12}] [Compare with Theorem $4.2$ in~\cite{Azarpanah}]. As in the space $C_{m^I}(X)$, ideals in $C(X)$ are never compact in $C_{U^I}(X)$ [Theorem ~\ref{Th3.23}(\ref{Th3.23(1)})] and the ideals contained in the ring $C_\psi(X)$ of all real valued continuous functions with pseudocompact support are the only candidates for Lindel\"{o}f ideals in $C_{U^I}(X)$ [Theorem~\ref{Th3.23}(\ref{Th3.23(2)})]. In~\cite{Bharati}, it is seen that a whole bunch of $Z$-ideals $I$ in $C(X)$, can give rise to identical $U_I$-topologies (respectively identical $m_I$-topologies). In contrast we observe in the present article that $U^I$-topologies on $C(X)$ (respectively $m^I$-topologies on $C(X)$) are uniquely determined by the ideal $I$ in $C(X)$ [Theorem~\ref{Th3.1}].\\
	In section~\ref{Sec4} of the present article on specializing $I=C(X)$ and therefore writing $C_U(X)$ instead of $C_{U^I}(X)$, we achieve characterization of two known classes of topological spaces $X$ viz strongly zero-dimensional space and pseudocompact weakly $P$-space in terms of the behavior of two chosen subsets $U(X)$ and $D(X)$ of the ring $C(X)$, in the space $C_U(X)$ [Theorem~\ref{Th4.2}, Theorem~\ref{Th4.3}]. Here $U(X)$ stands for the set of all units in $C(X)$ and $D(X)$, the collection of all zero-divisors in $C(X)$, including $0$. We further observe that the closure of the ideal $C_K(X)$ of all real valued continuous functions with compact support in the space $C_U(X)$ is precisely the set $\{f\in C(X):f^*(\beta X\setminus X)=\{0\}\}$, here $f^*:\beta X\to \mathbb{R}\cup\{\infty\}$ is the well known Stone-extension of the function $f$. This leads to the fact that the closure of $C_K(X)$ in $C_U(X)$ is the familiar ring $C_\infty(X)$ of all functions in $C(X)$ which vanishes at infinity [Remark~\ref{Rem4.9}]. We would like to point out at this moment, that the same proposition is very much there in the celebrated monograph [Theorem $3.17$,~\cite{Rudin}] but with the additional hypothesis that $X$ is locally compact. We also prove that the closure of the ideal $C_\psi(X)$ of all functions with pseudocompact support in the space $C_U(X)$ equals to the set $\{f\in C(X):f^*(\beta X\setminus\upsilon X)=\{0\}\}$ [Theorem~\ref{Th4.11}]. This ultimately leads to the proposition that the closure of $C_\psi(X)$ in $C_U(X)$ is the ring $C^\psi_\infty(X)=\{f\in C(X):\forall n\in\mathbb{N},\{x\in X:|f(x)|\geq\frac{1}{n}\}\text{ is pseudocompact}\}$. This last ring is called the pseudocompact analogue of the ring $C_\infty(X)$ and is initiated in~\cite{Acharyya}. The closure of $C_K(X)$ is $C_\infty(X)$ and that of $C_\psi(X)$ is $C^\psi_\infty(X)$ (in the space $C_U(X)$). These two apparently distinct facts are put on a common setting in view of the following result, which we establish at the end of this article. If $\mathscr{P}$ is an ideal of closed sets in $X$, in the sense that whenever $E,F\in\mathscr{P}$, then $E\cup F\in\mathscr{P}$ and $E\in\mathscr{P}$ and $C$, a closed set in $X$ with $C\subset E$ implies that $C\in\mathscr{P}$, then set $C_\mathscr{P}(X)=\{f\in C(X):\text{the support of }f\in\mathscr{P}\}$ and $C_\infty^\mathscr{P}(X)=\{f\in C(X):\forall n\in\mathbb{N}, \{x\in X:|f(x)|\geq\frac{1}{n}\}\in\mathscr{P}\}$. It is proved that the closure of $C_\mathscr{P}(X)$ in $C_U(X)$ is $C_\infty^\mathscr{P}(X)$ [Theorem~\ref{Th4.13}(\ref{Th4.13(2)})].\\
	Incidentally we establish a formula for the closure of $C_\mathscr{P}(X)$ in the space $C(X)$ equipped with $m$-topology. In fact we prove that the closure of $C_\mathscr{P}(X)$ in the $m$-topology $\equiv cl_mC_\mathscr{P}(X)=\{f\in C(X):f.g\in C^\mathscr{P}_\infty(X)\text{ for each }g\in C(X)\}$, Theorem~\ref{Th4.14}(\ref{Th4.14(3)}). With the special choice $\mathscr{P}\equiv$ the ideal of all compact sets in $X$ this formula reads $cl_m C_K(X)=\bigcap\limits_{p\in \beta X-X}M^p$. This last result is precisely the Proposition $5.6$~\cite{Azarpanah}. We conclude this article with a characterization of pseudocompact spaces via denseness of ideal $C_\mathscr{P}(X)$ in $C^\mathscr{P}_\infty(X)$ in the $m$-topology.
	\section{Answer to a few open problems concerning $U_I$-topologies and $m_I$-topologies on $C(X)$}\label{Sec2}
	At the very outset we need to explain a few notations. For each point $p\in\beta X,~ M^p=\{f\in C(X): p\in cl_{\beta X}Z(f)\}$, which is a maximal ideal in $C(X)$ and $O^p=\{f\in C(X):cl_{\beta X}Z(f)\text{ is a neighborhood of }p\text{ in the space }\beta X\}$, a well-known $Z$-ideal in $C(X)$. For each subset $A$ of $\beta X$, we prefer to write $M^A$ instead of $\bigcap\limits_{p\in A}M^p$. Analogously we write $O^A=\bigcap\limits_{p\in A} O^p$. We reproduce the following results from~\cite{Bharati}, to make the paper self-contained.
	\begin{theorem}[This is Theorem 4.10 in~\cite{Bharati}]
		If $A$ is a closed subset of $\beta X$, then $[M^A]=\{I\in\mathcal{I}:O^A\subseteq I\subseteq M^A\}$.
	\end{theorem}
	\begin{theorem}[This is Theorem 4.21 in~\cite{Bharati}]
		For any closed subset $A$ of $\beta X$, $[[M^A]]= \{I\in\mathcal{I}:O^A\subseteq I\subseteq M^A\}$.
	\end{theorem}
	We are going to establish a generalized version of each of the last two Theorems. We need the following subsidiary fact for that purpose.
	\begin{theorem}
		For any subset $A$ of $\beta X$, $M^A=M^{\overline{A}}$, here $\overline{A}=cl_{\beta X}A$
	\end{theorem}
	\begin{proof}
		It is trivial that $M^{\overline{A}}\subseteq M^A$. To prove the reverse containment let $f\in M^A$. This means that for each $p\in A, f\in M^p$. It follows from Gelfand-Kolmogorov Theorem [Theorem $7.3$,~\cite{Gillman}] that $A\subseteq cl_{\beta X}Z(f)$. This implies that $\overline{A}\subseteq cl_{\beta X}Z(f)$. We use Gelfand-Kolmogorov Theorem once again to conclude that $f\in M^{\overline{A}}$. Then $M^A\subseteq M^{\overline{A}}$.
	\end{proof}
	Theorem $2.1$ (respectively Theorem $2.2$) in conjunction with Theorem $2.3$ yields the following two Theorems almost immediately:
	\begin{theorem}
		For any subset $A$ of $\beta X$, $[M^A]=\{I\in\mathcal{I}:O^{\overline{A}}\subseteq I\subseteq M^A\}$.
	\end{theorem}
	\begin{theorem}
		If $A\subset\beta X$, then $[[M^A]]=\{I\in\mathcal{I}:O^{\overline{A}}\subseteq I\subseteq M^A\}$.
	\end{theorem}
	We want to recall at this moment that given a $Z$-ideal $I$ in $C(X)$, there always exists a set of maximal ideals $\{M^p:p\in A\}$ each containing $I$, $A$ being a suitable subset of $\beta X$ for which we can write: $[I]=[M^A]=[[I]]$ [This is proved in Theorem $4.9$ and Theorem $4.20$ in~\cite{Bharati}]. In view of this fact, we can make the following comments:
	\begin{remark}
		Each equivalence class $[I]$ in the quotient set $\mathcal{I}\slash\sim$ has a largest as well as a smallest member. This answers question $4.26$ raised in~\cite{Bharati} affirmatively.
	\end{remark}
	\begin{remark}
		Each equivalence class $[[I]]$ in the quotient set $\mathcal{I}\slash\approx$ has a largest as well as a smallest member [This answers question $4.27$ in~\cite{Bharati}].
	\end{remark}
	\begin{remark}
		For each $Z$-ideal $I$ in $C(X)$, $[I]=[[I]]$. Essentially this means that $\sim$ and $\approx$ are two identical binary relations on $\mathcal{I}$ [This answers question $4.25$ in \cite{Bharati} negatively].
	\end{remark}
	\section{$U^I$-topologies versus $m^I$-topologies on $C(X)$}\label{Sec3}
	We begin with the following simple result which states that the assignment: $I\to U^I$ is a one-one map.
	\begin{theorem}\label{Th3.1}
		Suppose $I$ and $J$ are two distinct ideals in $C(X)$. Then $U^I$-topology is different from $U^J$-topology.
	\end{theorem}
	\begin{proof}
		Without loss of generality, we can choose a function $g\in I\setminus J$ such that $|g(x)|<1$ for each $x\in X$. Clearly $\widetilde{B}(g,J,1)$ is an open set in the $U^J$-topology. We assert that this set is not open in the $U^I$-topology. If possible let $\widetilde{B}(g,J,1)$ be open in the $U^I$-topology. Then there exists $\epsilon>0$ in $\mathbb{R}$ such that $\widetilde{B}(g,I,\epsilon)\subseteq\widetilde{B}(g,J,1)$. Since $g+\frac{\epsilon}{2}g\in \widetilde{B}(g,I,\epsilon)$, this implies that $g+\frac{\epsilon}{2}g\in \widetilde{B}(g,J,1)$. It follows $g+\frac{\epsilon}{2}g-g\in J$, i.e., $\frac{\epsilon}{2}g\in J$, a contradiction to the initial choice that $g\notin J$. 
	\end{proof}
	\begin{remark}\label{Rem3.2}
		A careful modification in the above chain of arguments yields that $\widetilde{B}(g,J,1)$, which is an open set in the $U^J$-topology (and therefore open in the $m^J$-topology) is not open in the $m^I$-topology. Therefore we can say that whenever $I$ and $J$ are distinct ideals in $C(X)$, it is the case that $m^I$-topology is different from $m^J$-topology.
	\end{remark}
	Like any homogeneous space, $C_{U^I}(X)$ (respectively $C_{m^I}(X)$) is either devoid of any isolated point or all the points of this space are isolated. The following theorem clarifies the situation.
	\begin{theorem}
		The following three statements are equivalent for an ideal $I$ in $C(X)$:
		\begin{enumerate}
			\item $C_{U^I}(X)$ is a discrete space.
			\item $C_{m^I}(X)$ is a discrete space.
			\item $I=(0)$.
		\end{enumerate}
	\end{theorem}
	\begin{proof}
		If $I=(0)$, then for each $f\in C(X)$, $\widetilde{B}(f,I,1)=\{f\}$ and therefore each point of $C_{U^I}(X)$ (and $C_{m^I}(X)$) is isolated. This settles the implication $(3)\implies(1)$ and $(3)\implies(2)$. $(1)\implies(2)$ is trivial because $m^I$-topology is finer than the $U^I$-topology. Suppose $(3)$ is false, i.e., $I\neq (0)$. Then the Remark~\ref{Rem3.2} and the implication $(3)\implies(2)$ imply that $m^I$-topology is different from the discrete topology. Thus $(2)\implies(3)$.
	\end{proof}
	It is a standard result in the study of function spaces that $C_U(X)$ is a topological vector space if and only if $X$ is pseudocompact [$2M6$,~\cite{Gillman}]. The following fact is a minor improvement of this result.
	\begin{theorem}
		For an ideal $I$ of $C(X)$, $C_{U^I}(X)$ is a topological vector space if and only if $I=C(X)$ and $X$ is pseudocompact.
	\end{theorem}
	\begin{proof}
		If $I=C(X)$, then $C_{U^I}(X)=C_U(X)$, which is a topological vector space if $X$ is pseudocompact as observed above. Conversely let $C_{U^I}(X)$ be a topological vector space and $f\in C(X)$. Then there exists $\epsilon>0$ in $\mathbb{R}$ such that $(-\epsilon,\epsilon)\times\widetilde{B}(f,I,\epsilon)\subseteq \widetilde{B}(0,I,1)$. This implies that $\frac{\epsilon}{2}f\in I$ and hence $f\in I$. Thus $I=C(X)$. Clearly then $U^I$-topology on $C(X)$ reduces to the $U$-topology. We can therefore say that $C_U(X)$ is a topological vector space. In view of the observations made above, it follows that $X$ is a pseudocompact space.
	\end{proof}
	The following proposition gives a set of conditions in which each implies the next
	\begin{theorem}\label{Th3.5}
		\hspace*{3cm}
		\begin{enumerate}
			\item The $U^I$-topology = the $m^I$-topology on $C(X)$.
			\item $C_{U^I}(X)$ is a topological ring.
			\item $I\subset C^*(X)$.
			\item $I\cap C^*(X)=I\cap C_\psi(X)$.
		\end{enumerate}
	\end{theorem}
	\begin{proof}
		\hspace*{3cm}
		\begin{enumerate}
			\item $\implies(2)$ is trivial because $C_{m^I}(X)$ is a topological ring.
			\item$\implies(3):$ Suppose $(2)$ holds but $I\not\subset C^*(X)$. Choose $f\in I$ such that $f\notin C^*(X)$. Since the product function
			\begin{alignat*}{1}
				C_{U^I}(X)\times C_{U^I}(X)&\to C_{U^I}(X)\\
				(g,h)&\mapsto g.h
			\end{alignat*} is continuous at the point $(0,f)$, we get an $\epsilon>0$ such that $\widetilde{B}(0,I,\epsilon)\times\widetilde{B}(f,I,\epsilon)\subseteq\widetilde{B}(0,I,1)$. Let $g=\frac{\epsilon.f}{2(1+|f|)}$. Then $g\in \widetilde{B}(0,I,\epsilon)$, this implies that $g.f\in\widetilde{B}(0,I,1)$ and hence $g(x).f(x)<1$ for each $x\in X$ i.e., for each $x\in X$, $\frac{\epsilon.f^2(x)}{2(1+|f(x)|)}<1$. Now since $f$ is an unbounded function on $X$, $|f(x_n)|\to\infty$ along a sequence $\{x_n\}_n$ in $X$. Consequently $\lim\limits_{n\to\infty}\frac{|f(x_n)|}{1+|f(x_n)|}=\lim\limits_{n\to\infty}[1-\frac{1}{1+|f(x_n)|}]=1$ and therefore there exists $k\in\mathbb{N}$ such that for all $n\geq k$, $\frac{|f(x_n)|}{1+|f(x_n)|}\geq\frac{3}{4}$. This implies that for each $n\geq k$, $\frac{\epsilon.|f(x_n)|}{2}.\frac{3}{4}\leq\frac{\epsilon.|f(x_n)|}{2}.\frac{|f(x_n)|}{1+|f(x_n)|}<1$ and so $\{|f(x_n)|:n\geq k\}$ becomes a bounded sequence in $\mathbb{R}$. This contradicts that $|f(x_n)|\to\infty$ as $n\to\infty$. Hence $I\subset C^*(X)$.
			\item  $\implies(4):$ Suppose $(3)$ holds. We need to show that $I\cap C^*(X)\subset I\cap C_\psi(X)$ (because $C_\psi(X)\subseteq C^*(X)$). Since $C_\psi(X)$ is the largest bounded ideal in $C(X)$ [Theorem $3.8$,~\cite{Azarpanah2}]. The condition $(3)$ implies that $I\subset C_\psi(X)$. Hence $I\cap C^*(X)=I=I\cap C_\psi(X)$.
		\end{enumerate}
	\end{proof}
	The statement $(4)$ may not imply the statements in the Theorem~\ref{Th3.5}. Consider the following example:
	\begin{example}
		Take $X=\mathbb{N},\ I=C_K(\mathbb{N})$. Then $U^I$-topology on $C(\mathbb{N})\subsetneqq m^I$-topology on $C(\mathbb{N})$.\\
		Proof of this claim: First observe that $C_K(\mathbb{N})\subset C^*(\mathbb{N})$. Now recall the function $j\in C^*(\mathbb{N})$ given by $j(n)=\frac{1}{n},\ n\in\mathbb{N}$. Then $\widetilde{B}(0,I,j)$ is an open set in $C(\mathbb{N})$ with $m^I$-topology. We assert that this set is not open in $C(\mathbb{N})$ with $U^I$-topology. Suppose otherwise, then there exists $\epsilon>0$ such that $0\in\widetilde{B}(0,I,\epsilon)\subset\widetilde{B}(0,I,j)$. Now there exists $k\in\mathbb{N}$ such that $\frac{2}{k}<\epsilon$. Let $f(n)=\begin{cases}
			\frac{\epsilon}{2}&\text{ when }n\leq k\\
			0&\text{ otherwise}
		\end{cases}$, then $f\in\widetilde{B}(0,I,\epsilon)$. But $f\notin\widetilde{B}(0,I,j)$
	\end{example}
	$Y\subset X$ is called a relatively pseudocompact or bounded subset of $X$ if for every $f\in C(X)$, $f(Y)$ is a bounded subset of $\mathbb{R}$. The previous Theorem is a special case of the more general Theorem, given below:
	\begin{theorem}\label{Th3.8}
		For a convex ideal $I$ of $C(X)$, $U^I$-topology $=m^I$-topology if and only if $X\setminus\bigcap Z[I]$ is a bounded subset of $X$.
	\end{theorem}
	\begin{proof}
		First let $X\setminus\bigcap Z[I]$ be bounded and $\widetilde{B}(f,I,u)$ be an open set in $m^I$-topology, where $f\in C(X)$ and $u$, positive unit in $C(X)$. Now $\frac{1}{u}$ is bounded in $X\setminus\bigcap Z[I]$, i.e., there exists $\lambda>0$ such that $\frac{1}{u(x)}<\lambda$ for all $x\in X\setminus\bigcap Z[I]\implies u(x)>\frac{1}{\lambda}$ for all $x\in X\setminus\bigcap Z[I]$. We claim that $\widetilde{B}(f,I,\frac{1}{\lambda})\subset\widetilde{B}(f,I,u)$. Consider $g\in\widetilde{B}(f,I,\frac{1}{\lambda})$. Then $|g-f|<\frac{1}{\lambda}$ and $g-f\in I$. Now for all $x\in\bigcap Z[I]$, $(g-f)(x)=0<u(x)$ and for all $x\in X\setminus\bigcap Z[I]$, $|g(x)-f(x)|<\frac{1}{\lambda}<u(x)$, i.e., $|g-f|<u$ on $X$. So $g\in\widetilde{B}(f,I,u)$. For the converse part, suppose $X\setminus\bigcap Z[I]$ is not a bounded subset of $X$. Then there exists a positive unit $u$ in $C(X)$ and a $C$-embedded copy of $\mathbb{N}\subset X\setminus\bigcap Z[I]$ on which $u\to 0$. Clearly $\widetilde{B}(0,I,u)$ is an open set in $C_{m^I}(X)$. We claim that $\widetilde{B}(0,I,u)$ is not open in the $U^I$-topology. If possible let there exists $\epsilon>0$ such that $0\in\widetilde{B}(0,I,\epsilon)\subset\widetilde{B}(0,I,u)$. Since $u(n)\to 0$ as $n\to\infty$ for $n\in\mathbb{N}$, there exists $k\in\mathbb{N}$ such that $u(k)<\frac{\epsilon}{2}$. As $\mathbb{N}\subset X\setminus\bigcap Z[I]$, there exists an $f(\geq 0)\in I$ such that $f(k)>0$. Since $\mathbb{N}$ is $C$-embedded in $X$, there exists $h(\geq 0)\in C(X)$ such that $h(k)=\frac{\epsilon}{2f(k)}$. Let $g=f.h$. Then $g\in I$ and  $g(k)=\frac{\epsilon}{2}$. Set $g'=g\wedge\frac{\epsilon}{2}$. Then $g'\leq g\implies g'\in I$, as $I$ is convex. Also $g'\leq\frac{\epsilon}{2}\implies g'\in\widetilde{B}(0,I,\epsilon)$ which further implies that $g'\in\widetilde{B}(0,I,u)\implies g'<u$. But $g'(k)=\frac{\epsilon}{2}>u(k)$ -- a contradiction.
	\end{proof}
	\begin{remark}
		With the special choice $I=C(X)$, the above Theorem reads: The $U$-topology $=m$-topology on $C(X)$ if and only if $X$ is pseudocompact. This is a standard result in the theory of rings of continuous function [see 2M6 and 2N~\cite{Gillman}].
	\end{remark}
	It is proved in~\cite{Azarpanah}, Proposition $2.2$ that if $I$ is an ideal in $C(X)$ then any ideal $J$ containing $I$ is clopen in $C_{m^I}(X)$ and also $C^*(X)\cap I$ is clopen in $C_{m^I}(X)$. These two facts can be deduced from the following proposition, because the $m^I$-topology is finer than the $U^I$-topology.
	\begin{theorem}\label{Th3.7}
		\hspace*{3cm}
		\begin{enumerate}
			\item If $J$ is any additive subgroup of $(C(X),+,.)$ containing the ideal $I$, then $J$ is a clopen subset of $C_{U^I}(X)$.\label{Th3.7a}
			\item For any ideal $I$ in $C(X)$, $I\cap C^*(X)$ is a clopen subset of $C_{U^I}(X)$.\label{Th3.7b}
		\end{enumerate}
	\end{theorem}
	\begin{proof}
		\hspace*{3cm}
		\begin{enumerate}
			\item Let $f\in J$. Then $f\in\widetilde{B}(f,I,1)\subset J$, because $g\in \widetilde{B}(f,I,1)\implies g-f\in I\subset J\implies g=f+(g-f)\in J$. Thus $J$ becomes open in $C_{U^I}(X)$. To prove that $J$ is also closed in this space let $f\notin J, f\in C(X)$. Then it is not at all hard to check that $\widetilde{B}(f,I,1)\cap J$=$\emptyset$ and hence $J$ is closed in $C_{U^I}(X)$.
			\item For any $f\in I\cap C^*(X)$, it is routine to check that $f\in \widetilde{B}(f,I,1)\subset I\cap C^*(X)$. Then $I\cap C^*(X)$ is open in $C_{U^I}(X)$. To settle the closeness of $I\cap C^*(X)$ in $C_{U^I}(X)$, we need to verify that for any $f\in C(X)\setminus(I\cap C^*(X))$, $\widetilde{B}(f,I,1)\cap I\cap C^*(X)=\emptyset$ and that verification is also routine.
		\end{enumerate}
	\end{proof}
	Before proceeding further we recall for any $f\in C(X)$ the map \begin{alignat*}{1}
		\phi_f:\mathbb{R}&\to C(X)\\
		r&\mapsto r.f
	\end{alignat*} already introduced in~\cite{Azarpanah2},~\cite{Azarpanah}.
	\begin{lemma}\label{Lem3.8}
		Let $I$ be an ideal in $C(X)$. Then for $f\in C(X)$, \begin{alignat*}{1}
			\phi_f:\mathbb{R}&\to C_{U^I}(X)\\
			r&\mapsto r.f
		\end{alignat*} is a continuous map if and only if $f\in I\cap C^*(X)$ $[$compare with an analogous fact in the $m^I$-topology: Lemma $3.1$ in~\cite{Azarpanah}$]$.
	\end{lemma}
	\begin{proof}
		First assume that $\phi_f$ is continuous, in particular at the point $0$. So there exists $\delta>0$ in $\mathbb{R}$ such that $\phi_f(-\delta,\delta)\subseteq\widetilde{B}(\phi_f(0),I,1)=\widetilde{B}(0,I,1)$. This implies that $\phi_f(\frac{\delta}{2})\in\widetilde{B}(0,I,1)$ and hence $|\frac{\delta}{2}f|< 1$ and $\frac{\delta}{2}f\in I$. Clearly then $f\in C^*(X)\cap I$. Conversely let $f\in C^*(X)\cap I$. Then $|f|<M$ on $X$ for some $M>0$ in $\mathbb{R}$. Choose $r\in\mathbb{R}$ and $\epsilon>0$ arbitrarily. Then it is not at all hard to check that $\phi_f(r-\frac{\epsilon}{M},r+\frac{\epsilon}{M})\subseteq\widetilde{B} (\phi_f(r),I,\epsilon)$. Then $\phi_f$ is continuous at $r$.
	\end{proof}
	\begin{corollary}
		For $f\in C(X)$, \begin{alignat*}{1}
			\phi_f:\mathbb{R}&\to C_U(X)\\
			r&\mapsto r.f
		\end{alignat*} is continuous if and only if $f\in C^*(X)$.
	\end{corollary}
	\begin{theorem}\label{Th3.13}
		The component of $0$ in $C_{U^I}(X)$ is $I\cap C^*(X)$.
	\end{theorem}
	\begin{proof}
		It follows from Lemma~\ref{Lem3.8} that $I\cap C^*(X)=\bigcup\limits_{f\in I\cap C^*(X)}\phi_f(\mathbb{R})$, a connected subset of $C_{U^I}(X)$. Since $I\cap C^*(X)$ is a clopen set in $C_{U^I}(X)$ (Theorem~\ref{Th3.7}(\ref{Th3.7b})), it is the case that $I\cap C^*(X)$ is the largest connected subset of $C_{U^I}(X)$ containing $0$. Hence $I\cap C^*(X)$ is the component of $0$ in $C_{U^I}(X)$.
	\end{proof}
	\begin{corollary}
		$C^*(X)$ is the component of $0$ in $C_U(X)$
	\end{corollary}
	To find out when does the space $C_{U^I}(X)$ become locally compact, we reproduce the Lemma $4.1(a)$ from the article~\cite{Azarpanah}:
	\begin{lemma}
		For any positive unit $u$ in $C(X)$ and for a finite subset $\{a_1,a_2,...,a_k\}$ of $X\setminus \bigcap Z[I]$, for each $i\in\{1,2,...,k\}$, there exists $t_i\in I$ such that $|t_i|<u,\ t_i(a_i)=\frac{1}{2}u(a_i)$ and $t_i(a_j)=0$ for $j\neq i$.
	\end{lemma}
	We will need the following special version of this Lemma.
	\begin{lemma}\label{Lem3.11}
		Suppose $\epsilon>0$ and $\{a_1,a_2,...,a_n\}$ is a finite subset of $X$$\setminus\bigcap Z[I]$. Then for each $i\in\{1,2,...,n\}$, there exists $t_i\in I$ such that $|t_i|<\epsilon,\ t_i(a_i)=\frac{1}{2}\epsilon$ and $t_i(a_j)=0$ for all $j\neq i$.
	\end{lemma}
	\begin{theorem}\label{Th3.12}
		For an ideal $I$ in $C(X)$, the following three statements are equivalent:
		\begin{enumerate}
			\item $C_{U^I}(X)$ is nowhere locally compact.
			\item $C_{m^I}(X)$ is nowhere locally compact.
			\item $X\setminus\bigcap Z[I]$ is an infinite set.
		\end{enumerate}
	\end{theorem}
	\begin{proof}
		The equivalence of the statements $(2)$ and $(3)$ is precisely Theorem $4.2$ in~\cite{Azarpanah}. So we shall establish the equivalence of $(1)$ and $(3)$. The proof for this later equivalence will be a close adaption of the proof of Theorem $4.2$ in~\cite{Azarpanah}. However we shall make a sketch of this proof in order to make the paper self contained. First assume that $X\setminus\bigcap Z[I]$ is an infinite set. If possible let $K$ be a compact subset of $C_{U^I}(X)$ with non-empty interior. Then there exists $f\in C(X)$ and $\epsilon>0$ in $\mathbb{R}$ such that $\widetilde{B}(f,I,\epsilon)\subseteq K$. The compactness of $K$ in $C_{U^I}(X)$ implies that $K\subseteq\bigcup\limits_{i=1}^n \widetilde{B}(g_i,I,\frac{\epsilon}{4})$ for a suitable finite subset $\{g_1,g_2,...,g_n\}$ of $K$. Since $X\setminus\bigcap Z[I]$ is an infinite set, we can pick up $(n+1)$-many distinct members $\{a_1,a_2,...,a_{n+1}\}$ from this set. On using Lemma~\ref{Lem3.11}, we can find out for each $i\in\{1,2,...,n+1\}$, a function $t_i\in I$ such that $|t_i|<\epsilon,\ t_i(a_i)= \frac{\epsilon}{2}$ and $t_i(a_j)=0$ if $j\neq i, j\in\{1,2,...,n+1\}$. Set $k_i=f+t_i,i=1,2,...,n+1$. Then for each $i=1,2,...,n+1, k_i\in \widetilde{B}(f,I,\epsilon)\subset K\subseteq\bigcup\limits_{i=1}^n \widetilde{B}(g_i,I,\frac{\epsilon}{4})$, so there exists distinct $p,q\in\{1,2,...,n+1\}$ for which $k_p$ and $k_q$ lie in $\widetilde{B}(g_i,I,\frac{\epsilon}{4})$ for some $i\in\{1,2,...,n\}$. This implies that $|k_p-k_q|<\frac{\epsilon}{2}$, while $|k_p(a_p)-k_q(a_p)|=\frac{\epsilon}{2}$ -- a contradiction. Thus $(3)\implies(1)$ is established. If $X\setminus\bigcap Z[I]$ is a finite set, say the set $\{b_1,b_2,...,b_k\}$, then by proceeding analogously as in the proof of Lemma $4.1(b)$ in~\cite{Azarpanah}, we can easily show that $\mathbb{R}^k$ is homeomorphic to the subspace $I$ of the space $C_{U^I}(X)$. From Theorem~\ref{Th3.7}(\ref{Th3.7a}), we get that $I$ is an open subspace of $C_{U^I}(X)$. Hence the space $C_{U^I}(X)$ becomes locally compact at each point on $I$. Consequently $C_{U^I}(X)$ is locally compact at each point on $X$ (Mind that $C_{U^I}(X)$ is a homogeneous space).
	\end{proof}
	\begin{corollary}
		$C_U(X)$ is nowhere locally compact if and only if $C_m(X)$ is nowhere locally compact if and only if $X$ is an infinite set.
	\end{corollary}
	A sufficient condition for the nowhere local compactness of $C_{U^I}(X)$ is given as follows:
	\begin{theorem}
		If $I\not\subset C^*(X)$, then $C_{U^I}(X)$ is nowhere locally compact $[$compare with an analogous fact concerning $C_{m^I}(X)$ in Corollary $4.4$~\cite{Azarpanah}$]$.
	\end{theorem}
	\begin{proof}
		It is clear that $I\not\subset C^*(X)\implies X\setminus\bigcap Z[I]$ is an infinite set. It follows from Theorem~\ref{Th3.12} that $C_{U^I}(X)$ is nowhere locally compact.
	\end{proof}
	The following simple example shows that the converse of the last statement is not true.
	\begin{example}
		Take $X=\mathbb{R}$ and $I=C_K(\mathbb{R})$. Then $I\subset C^*(X)$, but $\bigcap\limits_{f\in C_K(\mathbb{R})}Z(f)=\emptyset$ and therefore $\mathbb{R}\setminus \bigcap Z[I]=\mathbb{R}=$ an infinite set. Hence from Theorem~\ref{Th3.12}, $C_{U^I}(\mathbb{R})$ is nowhere locally compact, though $I\subset C^*(\mathbb{R})$.
	\end{example}
	For an essential ideal $I$ in $C(X)$ [$I$ is called an essential ideal in $C(X)$ if $I\neq(0)$ and every non-zero ideal in $C(X)$ cuts $I$ non-trivially], the following fact is a simple characterization of nowhere local compactness of $C_{U^I}(X)$.
	\begin{theorem}
		Let $I$ be an essential ideal in $C(X)$. Then $C_{U^I}(X)$ is nowhere locally compact if and only if $X$ is an infinite set.
	\end{theorem}
	\begin{proof}
		For the essential ideal $I$ in $C(X)$, $\bigcap Z[I]$ is nowhere dense [Proposition $2.1$,~\cite{Azarpanah1}] and hence $cl_X(X\setminus\bigcap Z[I])=X$. The desired result follows on using Theorem~\ref{Th3.12} in a straightforward manner.
	\end{proof}
	We would like to point out at this moment that it is mentioned in \cite{Azarpanah2} [the proof of the implication relation $(f)\implies(b)$ in Proposition $3.14$] and also in  \cite{Azarpanah} (the statement lying between Corollary $3.5$ and Corollary $3.6$) that whenever $C_\psi(X)\neq\{0\}$, then it is an essential ideal in $C(X)$. The following counterexample shows that there exists a non-zero $C_\psi(X)$ in $C(X)$, which is not an essential ideal in $C(X)$.
	\begin{example}
		Consider the following subspace of $\mathbb{R}: X=\{0\}\cup \{x\in\mathbb{R}:x\text{ is rational and }1\leq x\leq 2\}$. Then $X$ is locally compact at the point $0$ and therefore $C_K(X)\neq\{0\}$, because for a space $Y$, $C_K(Y)$ is $\{0\}$ if and only if $Y$ is nowhere locally compact [This follows on adapting the arguments in $4D2$~\cite{Gillman}, more generally for a nowhere locally compact space $Y$ instead of $\mathbb{Q}$ only]. Since $X$ is a metrizable space, there is no difference between compact and pseudocompact subsets of $X$. Hence $C_\psi(X)=C_K(X)\neq\{0\}$. It is clear that if $f\in C_K(X)$, then $f$ vanishes at each point on $X\setminus\{0\}$. Consequently $\bigcap\limits_{f\in C_K(X)}Z(f)=[1,2]\cap\mathbb{Q}$, which being a non-empty clopen set in the space $X$ is not nowhere dense. Hence on using Proposition $2.1$ in~\cite{Azarpanah1}, $C_K(X)$ is not an essential ideal in $C(X)$. 
	\end{example}
	It is proved in Proposition $4.6$ in~\cite{Azarpanah} that a non-zero ideal $I$ in $C(X)$ is never compact in $C_{m^I}(X)$ and if such an $I$ is Lindel\"{o}f, then $I\subseteq C_\psi(X)$. These two facts can be deduced from the following proposition, because the $m^I$-topology is finer than the $U^I$-topology.
	\begin{theorem}
		Let $J$ be a non-zero ideal in $C(X)$. Then:
		\begin{enumerate}\label{Th3.23}
			\item $J$ is not compact in $C_{U^I}(X)$.\label{Th3.23(1)}
			\item If $J$ is Lindel\"{o}f in $C_{U^I}(X)$, then $J\subseteq C_\psi(X)$.\label{Th3.23(2)}
		\end{enumerate}
	\end{theorem}
	We omit the proof of this Theorem, because this can be done on closely following the arguments for the proof of Proposition $4.6$ in~\cite{Azarpanah}.
	\section{A few special properties for the space $C_U(X)$}\label{Sec4}
	If $U(X)$ is dense in $C_m(X)$, then it is plain that $U(X)$ is dense in $C_U(X)$, because the $U$-topology on $C(X)$ is weaker than the $m$-topology. We are going to show that the converse of this statement is true. We recall in this context that a space $X$ is strongly zero-dimensional if given a pair of completely separated sets $K$ and $W$ in $X$, there exists a clopen set $C'$ such that $K\subseteq C'\subseteq X\setminus W$. Equivalently $X$ is strongly zero-dimensional if and only if given a pair of disjoint zero-sets $Z$ and $Z'$ in $X$, there exists a clopen set $C$ in $X$ such that $Z\subseteq C\subseteq X\setminus Z'$. The following lemma gives a sufficient condition for the strongly zero-dimensionality of $X$. 
	\begin{lemma}\label{Lem4.1}
		Let $U(X)$ be dense in $C_U(X)$. Then $X$ is strongly zero-dimensional.
	\end{lemma}
	\begin{proof}
		Let $Z_1, Z_2$ be disjoint zero-sets in $X$. Then there exists $f\in C(X)$ such that $|f|\leq 1$, $f(Z_1)=\{-1\}$ and $f(Z_2)=\{1\}$. Since $U(X)$ is dense in $X$, we can find out a member $u\in\widetilde{B}(f,\frac{1}{2})\cap U(X)$. Let $C=\{x\in X:u(x)<0\}$. Then $C$ is a clopen set in $X$, $Z_1\subseteq C\subseteq X\setminus Z_2$. Thus $X$ becomes strongly zero-dimensional.
	\end{proof}
	\begin{theorem}\label{Th4.2}
		The following statements are equivalent for a space $X$.
		\begin{enumerate}
			\item $X$ is strongly zero-dimensional.
			\item $U(X)$ is dense in $C_U(X)$.
			\item $U(X)$ is dense in $C_m(X)$.
		\end{enumerate}
	\end{theorem}
	\begin{proof}
		The equivalence of $(1)$ and $(3)$ is precisely the Proposition 5.1 in~\cite{Azarpanah}. This combined with Lemma~\ref{Lem4.1} finishes the proof.
	\end{proof}
	Let $U^*(X)=\{u\in C(X):|u|>\lambda\text{ for some }\lambda>0\}$.
	\begin{theorem}\label{Th4.3}
		$cl_UD(X)(\equiv\text{ the closure of }D(X)\text{ in the space }C_U(X))=C(X)\setminus U^*(X)$ $[$compare with the fact: $cl_mD(X)=C(X)\setminus U(X)$ in Proposition $5.2$ in~\cite{Azarpanah}$]$.
	\end{theorem}
	\begin{proof}
		It is easy to check that $U^*(X)$ is open in $C_U(X)$ because choosing $u\in  C^*(X)$, we have $|u|>\lambda$ for some $\lambda>0$, this implies that $\widetilde{B}(u,\frac{\lambda}{2})\subseteq U^*(X)$ (We are simply writing $\widetilde{B}(u,\frac{\lambda}{2}$) instead of $\widetilde{B}(u,C(X),\frac{\lambda}{2})$). Since $D(X)\cap U(X)=\emptyset$, in particular $D(X)\cap U^*(X)=\emptyset$, it follows therefore that $cl_UD(X)\subseteq C(X)\setminus U^*(X)$. To prove the reverse inclusion relation, let $f\in C(X)\setminus U^*(X)$ and $\epsilon>0$ be preassigned. We need to show that $\widetilde{B}(f,\epsilon)\cap D(X)\neq\emptyset$. For that purpose define as in the proof of Proposition $5.2$ in~\cite{Azarpanah}. 
		$$ h(x)=
		\begin{cases}
			f(x)+\frac{\epsilon}{2} &\ \text{if}\ f(x)\leq -\frac{\epsilon}{2}	\\
			0&\ \text{if}\ |f(x)|\leq \frac{\epsilon}{2} \\
			f(x)-\frac{\epsilon}{2}&\ \text{if}\ f(x)\geq \frac{\epsilon}{2}	\\
		\end{cases}$$
		Then $h\in C(X)$. Since $f\notin U^*(X),\ f$ takes values arbitrarily near to zero on $X$. Therefore there exists $x\in X$ for which $|f(x)|<\frac{\epsilon}{2}$. This implies that $int_XZ(h)\neq\emptyset$. Thus $h\in D(X)$ and surely $|h-f|<\epsilon$. Therefore $h\in \widetilde{B}(f,\epsilon)\cap D(X)$.
	\end{proof}
	\begin{definition}
		We call a space $X$, a weakly $P$-space if whenever $f\in C(X)$ is such that $f$ takes values arbitrarily near to zero, then $f$ vanishes on some neighborhood of a point in $X$, i.e., $int_XZ(f)\neq\emptyset $.
	\end{definition}
	It is clear that every weakly $P$-space is an almost $P$-space and is pseudocompact. The following proposition is a characterization of weakly $P$-spaces.
	\begin{theorem}
		$X$ is a weakly $P$-space if and only if $D(X)$ is closed in $C_U(X)$ $[$Compare with the Proposition 5.2 in~\cite{Azarpanah}$]$.
	\end{theorem}
	\begin{proof}
		Let $X$ be a weakly $P$-space. This means that if $f\in C(X)$ is not a zero-divisor, then it is bounded away from zero, i.e., $f\in U^*(X)$. Thus $C(X)\setminus D(X)\subseteq U^*(X)$. The implication relation $U^*(X)\subseteq C(X)\setminus D(X)$ is trivial. Therefore $C(X)\setminus D(X)=U^*(X)$ and thus $D(X)=C(X)\setminus U^*(X)$. It follows from Theorem~\ref{Th4.3} that $D(X)$ is closed in $C_U(X)$. Conversely let $D(X)$ be closed in $C_U(X)$. Then this implies by Theorem~\ref{Th4.3} that $D(X)=C(X)\setminus U^*(X)$. Now let $f\in C(X)$ be such that $f$ takes values arbitrarily near to zero. We need to show that $int_{X}Z(f)\neq\emptyset$. If possible let $int_XZ(f)=\emptyset$. Then $f\notin D(X)$ and hence $f\in U^*(X)$ -- a contradiction.
	\end{proof}
	The next proposition shows that weakly $P$-spaces are special kind of almost $P$-spaces.
	\begin{theorem}
		$X$ is a weakly $P$-space if and only if it is pseudocompact and almost $P$.
	\end{theorem}
	\begin{proof}
		It is already settled that a weakly $P$-space is pseudocompact and almost $P$. Conversely let $X$ be pseudocompact and almost $P$. Suppose $f\in C(X)$ takes values arbitrarily near to zero on $X$. Then $f$ must attain the value $0$ at some point on $X$ because $X$ is pseudocompact. Thus $Z(f)\neq\emptyset$ and hence due to the almost $P$ property of $X$, we shall have $int_XZ(f)\neq \emptyset$. Therefore $X$ becomes weakly $P$.
	\end{proof}
	\begin{remark}
		$D(X)$ is closed in $C_U(X)$ if and only if $X$ is a pseudocompact almost $P$-space.
	\end{remark}
	There are enough examples of pseudocompact almost $P$-spaces. Indeed, if $X$ is a locally compact realcompact space, then  $\beta X\setminus X$ is a compact almost $P$-space [Lemma $3.1$,~\cite{Fine}].\\
	In what follows we compute the closure of a few related ideals in the ring $C(X)$.
	\begin{theorem}\label{Th4.8}
		$cl_UC_K(X)\ (\equiv\text{ the closure of }C_K(X)\text{ in the space }C_U(X) )=\{f\in C(X):f^*(\beta X\setminus X)=\{0\}\}$. 
	\end{theorem}
	\begin{proof}
		Set for each $p\in\beta X,\ \widetilde{M}^p=\{f\in C(X): f^*(p)=0\}$. Since $f\in M^p\implies p\in cl_{\beta X}Z(f)$ (Gelfand-Kolmogoroff Theorem) $\implies f^*(p)=0$, it follows that $M^p\subseteq\widetilde{M}^p$ for each $p\in\beta X$. Furthermore, $\widetilde{M}^p=\{f\in C(X):|M^p(f)|=0\text{ or infinitely small in the residue class field }C(X)\slash M^p\}$ [Theorem $7.6(b)$,~\cite{Gillman}]. It is well-known [vide~\cite{Plank}, Lemma $2.1$] that $cl_U M^p=\{f\in C(X):|M^p(f)|=0 \text{ or infinitely small}\}$. Hence we get that $M^p\subseteq cl_U M^p=\widetilde{M}^p$ for each $p\in\beta X$. Therefore $C_K(X)=\bigcap\limits_{p\in \beta X\setminus X}O^p$ [$7E$~\cite{Gillman}] $\subseteq\bigcap\limits_{p\in\beta X\setminus X}M^p\subseteq\bigcap\limits_{p\in \beta X\setminus X}\widetilde{M}^p=$ the intersection of a family of closed sets in $C_U(X)\equiv$ a closed set in $C_U(X)$. This implies that $cl_UC_K(X) \subseteq\bigcap\limits_{p\in\beta X\setminus X}\widetilde{M}^p=\{f\in C(X):f^*(\beta X\setminus X)=0\}$. To prove the reverse inclusion relation, let $f\in\bigcap\limits_{p\in \beta X\setminus X}\widetilde{M}^p $. Thus $f^*(\beta X\setminus X)=\{0\}$. Consequently then $f$ becomes bounded on $X$, for if $f$ is unbounded on $X$, then there exists a copy of $\mathbb{N}$, $C$-embedded in $X$ for which $\lim\limits_{n\to\infty}|f(x)|=\infty$. Surely then $cl_{\beta X}\mathbb{N}=\beta\mathbb{N}$ and so $cl_{\beta X}\mathbb{N}\setminus\upsilon X\supseteq\beta\mathbb{N}\setminus\mathbb{N}$ [We use the fact that a countable $C$-embedded subset of a Tychonoff space is a closed subset of it $3B3$~\cite{Gillman}]. Choose a point $p\in\beta\mathbb{N}\setminus\mathbb{N}$, It is clear that $f^*(p)=\infty$ -- a contradiction. Thus $f\in C^*(X)$ and we can write $f^\beta(\beta X\setminus X)=\{0\}$, here $f^\beta:\beta X\to \mathbb{R}$ is the Stone-extension of $f\in C^*(X)$. So $\beta X\setminus X \subseteq Z_{\beta X}(f^\beta)$, the zero set of $f^\beta$ in the space $\beta X$. Choose $\epsilon>0$. We claim that $\widetilde{B}(f,\epsilon)\cap C_K(X)\neq\emptyset$ and we are done.\\
		Proof of the claim: Define a function $h:X\to\mathbb{R}$ as follows 
		\[h(x)=\begin{cases}
			f(x)+\frac{\epsilon}{2} &\ \text{if} \ f(x)\leq -\frac{\epsilon}{2}\\
			0 &\ \text{if} \ -\frac{\epsilon}{2}\leq f(x)\leq \frac{\epsilon}{2}\\
			f(x)-\frac{\epsilon}{2} &\ \text{if} \ f(x)\geq\frac{\epsilon}{2}\\
		\end{cases}\]  Then $h\in C^*(X)$ and $|h(x)-f(x)|<\epsilon$ for each $x\in X$, i.e., $h\in\widetilde{B}(f,\epsilon)$. To complete this theorem, it remains only to check that $h\in C_K(X)$. Indeed let $g=(|f|\wedge\frac{\epsilon}{2}) -\frac{\epsilon}{2}$. Then $Z(f)\subseteq X\setminus Z(g)$ and $X\setminus Z(g)\subseteq Z(h)$ and hence $g.h=0$. Since the map \begin{alignat*}{2} C^*(X)&\to C(\beta X)\\
			k&\mapsto k^p
		\end{alignat*} is a lattice isomorphism, from the definition of $g$, we can at once write: $g^\beta=(|f|^\beta\wedge\frac{\epsilon}{2})-\frac{\epsilon}{2}$ and $g^\beta.h^\beta=0$. Consequently then, $\beta X\setminus Z_{\beta X}(g^\beta)\subseteq Z(h^\beta)$ and also, $Z_{\beta X}(g^\beta)\subseteq\beta X\setminus Z_{\beta X}(f^\beta)$. This shows that $Z(h^\beta)$ is a neighborhood of $\beta X\setminus X$. It follows from $7E$~\cite{Gillman} that $h\in C_K(X)$.
	\end{proof}
	\begin{remark}\label{Rem4.9}
		It is a standard result in the theory of rings of continuous functions that the complete list of free maximal ideals in $C^*(X)$ is given by $\{M^{*p}:p\in \beta X\setminus X\}$, where $M^{*p}=\{h\in C^*(X):h^\beta(p)=0\}$ [Theorem $7.2$,~\cite{Gillman}]. It is also well-known that [vide $7F1$,~\cite{Gillman}], $\bigcap\limits_{p\in \beta X\setminus X}M^{*p}=C_\infty(X)$. Hence we can ultimately write $cl_U {C_K(X)}=C_\infty(X)$.
	\end{remark}
	\begin{remark}
		We can show that for a well chosen collection of naturally existing spaces, $C_K(X)$ is not dense in $C_\infty(X)$ in the $m$-topology on $C(X)$. Indeed let $X$ be a locally compact, $\sigma$-compact non compact space (say $X=\mathbb{R}^n, n\in\mathbb{N}$). Since every $\sigma$-compact space is realcompact, it follows from Theorem $8.19$ in~\cite{Gillman} that $C_K(X)=\bigcap\limits_{p\in \beta X\setminus X}M^p$. Incidentally it is proved in Proposition $5.6$ in~\cite{Azarpanah} that $cl_mC_K(X)(\equiv\text{ the closure of }C_K(X)\text{ in the space }C_m(X))=\bigcap\limits_{p\in \beta X\setminus X}M^p$. Thus $C_K(X)$ is closed in $C_m(X)$. On the other hand it follows from $7F3$~\cite{Gillman} that with the above mentioned condition on $X$, the intersection of all free maximal ideals in $C(X)\subsetneq$ the intersection of all free maximal ideals in $C^*(X)$. Therefore $C_K(X)\subsetneqq C_\infty(X)$ and hence $C_K(X)$ is not closed in $C_\infty(X)$ in the space $C(X)$ in the $m$-topology.
	\end{remark}
	\begin{theorem}\label{Th4.11}
		$cl_U{C_\psi(X)}\ (\equiv \text{ the closure of }C_\psi(X)\text{ in }C_U(X)) =\bigcap\limits_{p\in\beta X\setminus X}\widetilde{M}^p=\{f\in C(X): f^*(\beta X\setminus\upsilon X)=\{0\}\}$.
	\end{theorem}
	\begin{proof}
		We shall follow closely the technique adopted to prove Theorem~\ref{Th4.8}. First recall the well-known fact: $C_\psi(X)=\bigcap\limits_{p\in\beta X\setminus \upsilon X}M^p$, Theorem $3.1$~\cite{Johnson}. It follows on adopting the chain of arguments in the first part of the proof of Theorem~\ref{Th4.8} that $cl_U{C_\psi(X)}\subseteq \bigcap\limits_{p\in\beta X\setminus X}\widetilde{M}^p=\{f\in C(X):f^*(\beta X\setminus\upsilon X)=\{0\}\}$. To prove the reverse inclusion relation, choose $f\in C(X)$ such that $f^*(\beta X\setminus\upsilon X)=\{0\}$, then it is not at all hard to prove that $f$ is bounded on $X$ and therefore we can rewrite as in the proof of Theorem~\ref{Th4.8} that $cl_UC_\psi(X)\subseteq\{f\in C^*(X):f^\beta (\beta X\setminus\upsilon X)=\{0\}\}$ and hence $\beta X\setminus\upsilon X\subseteq Z_{\beta X}(f^\beta)$. Next choosing $\epsilon>0$ and proceeding exactly as in the proof of Theorem~\ref{Th4.8}, thereby defining the bounded continuous function $h:X\to\mathbb{R}$ verbatim. We can easily check that $h\in\widetilde{B}(f,\epsilon)$. In the next stage we set as in the proof of Theorem~\ref{Th4.8}, $g=(|f|\wedge \frac{\epsilon}{2})-\frac{\epsilon}{2}$ and ultimately reach the inequality: \[ \beta X \setminus \upsilon X\subseteq Z_{\beta X}(f^\beta)\subseteq \beta X\setminus Z_{\beta X}(g^\beta) \subseteq Z_{\beta X}(h^\beta)\ \ \ \ldots (1)\] To complete this theorem, it remains to check that $h\in C_\psi(X)$. Since $C_\psi(X)=\bigcap\limits_{p\in \beta X\setminus\upsilon X}M^p$, it is therefore sufficient to show that (in view of Gelfand-Kolmogoroff Theorem~\cite{Gillman}), for each point $p\in\beta X\setminus \upsilon X,\ p\in cl_{\beta X}Z(h)$. For that purpose let $U$ be an open neighborhood of $p$ in $\beta X$. Then $V=\beta X\setminus Z_{\beta X }(g^\beta)\cap U$ is an open neighborhood of $p$ in $\beta X$ (we exploit the inequality $(1)$). Therefore $V\cap X\neq\emptyset$. But from $(1)$ we get that $V\cap X\subseteq Z(h)$. hence $Z(h)\cap U\neq\emptyset$. Thus each open neighborhood of $p$ in $\beta X$ cuts $Z(h)$ and therefore $p\in cl_{\beta X}Z(h)$.
	\end{proof}
	For notational convenience let us write for $f\in C(X)$ and $n\in\mathbb{N}$, $A_n(f)=\{x\in X:|f(x)|\geq\frac{1}{n}\}$. Since a support, i.e., a set of the form $cl_X(X\setminus Z(k)),\ k\in C(X)$ is pseudocompact if and only if it is bounded meaning that each $h\in C(X)$ is bounded on $cl_X(X\setminus Z(k))$ [Theorem $2.1$,~\cite{Mandelker}]. We rewrite: $C_\infty^\psi(X)=\{f\in C(X): A_n(f)\text{ is bounded for each }n\in \mathbb{N}\}$ [see~\cite{Acharyya} in this connection]. The following  result relates this ring with $C_\psi(X)$.
	\begin{theorem}\label{Th4.12}
		$C_\psi^\infty(X)=cl_UC_\psi(X)$.
	\end{theorem}
	\begin{proof}
		In view of Theorem~\ref{Th4.11}, it amounts to showing that $C_\psi^\infty (X)=\{f\in C(X):f^*(\beta X\setminus\upsilon X)=0\}$. For that we make the elementary but important observation that $C_\psi^\infty(X)\subseteq C^*(X)$. First assume that $f\in C(X)$ and $f^*(\beta X\setminus\upsilon X)=\{0\}$, i.e., $f^\beta(\beta X\setminus\upsilon X)=\{0\}$. Choose $n\in \mathbb{N}$ arbitrarily, we shall show that $A_n(f)$ is bounded. For that purpose select $g\in C(X)$ at random. Now by abusing notation we write $ A_n(f^\beta)=\{p\in \beta X:|f^\beta(p)|\geq\frac{1}{n}\}$. Then it is clear that $A_n(f^\beta)\subseteq \upsilon X$ and surely $A_n(f^\beta)$ is compact. It follows that for the function $g^*:\beta X\to\mathbb{R}\cup\{\infty\},\ g^*(A_n(f^\beta))$ is compact subset of $\mathbb{R}$. In particular we can say that $g$ is bounded on $A_n(f)$, which we precisely need. Thus it is proved that $\{f\in C(X):f^*(\beta X\setminus \upsilon X)=\{0\}\}\subseteq C_\psi^\infty(X)$. To prove the reverse containment, let $f\in C^*(X)$ and $f^*(\beta X\setminus\upsilon X)\neq\{0\}$. Without loss of generality we can take $f\geq 0$ on $X$, this means that there exists $p\in\beta X\setminus\upsilon X$ and $n\in\mathbb{N}$, for which $f^\beta(p)>\frac{1}{n}$. Hence there exists an open neighborhood $U$ of $p$ in $\beta X$ for which $f^\beta>\frac{1}{n}$ on th entire $U$. It follows that $p\in cl_\beta A_n(f)$. On the other hand, since $p\notin\upsilon X$, there exists $g\in C(X)$ such that $g^*(p)=\infty$. These two facts together imply that $g$ is unbounded on $A_n(f)$. Hence $A_n(f)$ is not pseudocompact and thus $f\notin C_\infty^\psi(X)$. The theorem is completely proved. 
	\end{proof}
	\begin{theorem}\label{Th4.13}
		Let $\mathscr{P}$ be an ideal of closed set in $X$. Then 
		\begin{enumerate}
			\item $C_\infty^\mathscr{P}(X)$ is a closed subset of $C_U(X)$.
			\item $cl_UC_\mathscr{P}(X)=C_\infty^\mathscr{P}(X)$.\label{Th4.13(2)}
		\end{enumerate}
	\end{theorem}
	\begin{proof}
		\hspace*{3cm}
		\begin{enumerate}
			\item Let us rewrite: $C_\infty^\mathscr{P}(X)=\{f\in C(X):\text{ for each } n\in\mathbb{N}, A_n(f)\in\mathscr{P}\}$. Suppose $f\in C(X)$ is such that $f\notin C_\infty^\mathscr{P}(X)$. Thus there exists $n\in\mathbb{N}$ such that $A_n(f)\notin\mathscr{P}$. We claim that $\widetilde{B}(f,\frac{1}{2n})\cap C_\infty^\mathscr{P}(X)=\emptyset$ and we are done. If possible let there exists $g\in\widetilde{B}(f,\frac{1}{2n}) \cap C_\infty^\mathscr{P}(X)$. Then $|g-f|<\frac{1}{2n}$ and $A_k(g)\in \mathscr{P}$ for each $k\in\mathbb{N}$. The first inequality implies that $|f|<|g|+\frac{1}{2n}$, which further implies that $A_n(f)\subseteq A_{2n}(g)$. This combined with $A_{2n}(g)\in \mathscr{P}$ yields, in view of the fact that $\mathscr{P}$ is an ideal of closed sets in $X$ that $A_n(f)\in \mathscr{P}$, a contradiction.
			\item Let $f\in C_\infty^\mathscr{P}(X)$ and $\epsilon>0$ in $\mathbb{R}$. Define as in the proof of Theorem~\ref{Th4.8}, a function $g:X\to\mathbb{R}$ as follows: \[ g(x)=\begin{cases}
				f(x)+\frac{\epsilon}{2} &\ if \ f(x)\leq -\frac{\epsilon}{2}\\
				0 &\ if \ -\frac{\epsilon}{2}\leq f(x)\leq \frac{\epsilon}{2}\\
				f(x)-\frac{\epsilon}{2} & \ if \ f(x)\geq \frac{\epsilon}{2}\\
			\end{cases} \]
			Then $g\in\widetilde{B}(f,\epsilon)$, we assert that $g\in C_\mathscr{P}(X)$ and therefore $\widetilde{B}(f,\epsilon)\cap C_\mathscr{P}(X)\neq\emptyset$ and we are done. Proof of the assertion: $X\setminus Z(g)\subseteq \{x\in X: |f(x)|\geq\frac{\epsilon}{2}\}$, this implies that: $cl_X(X\setminus Z(g)) \subseteq\{x\in X:|f(x)|\geq\frac{\epsilon}{2}\}$. Since $f\in C_\infty^\mathscr{P}(X)$, if follows that $\{x\in X: |f(x)|\geq\frac{\epsilon}{2}\}\in\mathscr{P}$ and hence $cl_X(X\setminus Z(g))\in\mathscr{P}$. Thus $g\in C_\mathscr{P}(X)$. 
		\end{enumerate}
	\end{proof}
	Set $I_\mathscr{P}=\{f\in C(X):f.g\in C^\mathscr{P}_\infty(X)\text{ for each }g\in C(X)\}$.
	\begin{theorem}\label{Th4.14}
		The following results hold:
		\begin{enumerate}
			\item $I_\mathscr{P}$ is an ideal in $C(X)$ with $C_\mathscr{P}(X)\subset I_\mathscr{P}\subset C^\mathscr{P}_\infty(X)$.
			\item $I_\mathscr{P}$ is closed in $C_m(X)$.
			\item $cl_mC_\mathscr{P}(X)=I_\mathscr{P}$.\label{Th4.14(3)}
			\item $I_\mathscr{P}=\bigcap\limits_{p\in F_\mathscr{P}}M^p$, where $F_\mathscr{P}=\{p\in\beta X:C_\mathscr{P}(X)\subset M^p\}$.
		\end{enumerate}
	\end{theorem}
	\begin{proof}
		\hspace*{3cm}
		\begin{enumerate}
			\item Let $f,g\in I_\mathscr{P}$ and $h\in C(X)$. Then $f.h, g.h\in C^\mathscr{P}_\infty(X)\implies (f+g)h\in C^\mathscr{P}_\infty(X)$, because $A_n(f.h+g.h)\subset A_{2n}(f.h)\cup A_{2n}(g.h)$ for each $n\in\mathbb{N}$. Also let $f\in I_\mathscr{P}$ and $g\in C(X)$. Consider any $h\in C(X)$. Then $g.h\in C(X)\implies f.g.h\in C^\mathscr{P}_\infty(X)\implies f.g\in I_\mathscr{P}$. Thus $I_\mathscr{P}$ is an ideal in $C(X)$. Clearly $C_\mathscr{P}(X)\subset I_\mathscr{P}\subset C^\mathscr{P}_\infty(X)$.
			\item Let $f\in C(X)$ such that $f\notin I_\mathscr{P}$. Then there exists $g\in C(X)$ such that $f.g\notin C^\mathscr{P}_\infty(X)$. Therefore there exists $p\in\mathbb{N}$ such that $A_p(f.g)\notin\mathscr{P}$. Let $u=\frac{1}{2p(1+|g|)}$. Then $u$ is a positive unit in $C(X)$. If possible let $h\in \widetilde{B}(f,u)\cap I_\mathscr{P}$. Then $|f-h|<u$ and $h\in I_\mathscr{P}$. Then $h.g\in C^\mathscr{P}_\infty(X)\implies A_n(h.g)\in\mathscr{P}$ for all $n\in\mathbb{N}$. Now $|f-h|<u\implies |f.g-h.g|<u|g|<\frac{1}{2p}\implies |f.g|<|h.g|+\frac{1}{2p}\implies A_p(f.g)\subset A_{2p}(h.g)\implies A_p(f.g)\in\mathscr{P}$ -- a contradiction. Therefore $\widetilde{B}(f,u)\cap I_\mathscr{P}=\emptyset$ and hence $I$ is closed in $C_m(X)$.
			\item Since $I_\mathscr{P}$ is closed in $C_m(X)$, it follows that $cl_mC_\mathscr{P}(X)\subseteqq I_\mathscr{P}$. Let $f\in I_\mathscr{P}$ and $u$ be any positive unit in $C(X)$. Define a function $g:X\to\mathbb{R}$ as follows: \[ g(x)=\begin{cases}
				f(x)+\frac{1}{2}u(x) &\ if \ f(x)\leq -\frac{1}{2}u(x)\\
				0 &\ if \ -\frac{1}{2}u(x)\leq f(x)\leq \frac{1}{2}u(x)\\
				f(x)-\frac{1}{2}u(x) & \ if \ f(x)\geq \frac{1}{2}u(x)\\
			\end{cases} \]
			Then $g\in\widetilde{B}(f,u)$ and $cl_X(X\setminus Z(g))\subset \{x\in X: |f(x)|\geq\frac{1}{2}u(x)\}$. Now $\frac{1}{u}\in C(X)$ and $f\in I_\mathscr{P}\implies\frac{f}{u}\in C^\mathscr{P}_\infty(X)\implies A_n(\frac{f}{u})\in\mathscr{P}$ for all $n\in\mathbb{N}$. It is clear that $A_2(\frac{f}{u})=\{x\in X: |f(x)|\geq\frac{1}{2}u(x)\}$ and so $cl_X(X\setminus Z(g))\in\mathscr{P}$, i.e., $g\in C_\mathscr{P}(X)$. Thus $\widetilde{B}(f,u)\cap C_\mathscr{P}(X)\neq\emptyset$, i.e., $f\in cl_mC_\mathscr{P}(X)$. So $I_\mathscr{P}\subseteqq cl_mC_\mathscr{P}(X)$.
			\item We know that the closure of an ideal $J$ of $C(X)$ in the $m$-topology is the intersection of all maximal ideal containing $J$ [$7Q2$~\cite{Gillman}]. Therefore $I_\mathscr{P}=cl_mC_\mathscr{P}(X)=\bigcap\limits_{p\in F_\mathscr{P}}M^p$.
		\end{enumerate}
	\end{proof}
	\begin{corollary}
		The closure of $C_K(X)$ in the $m$-topology is the ideal $\{f\in C(X):fg\in C_\infty(X)\text{ for each }g\in C(X)\}$. When $\mathscr{P}$ is the ideal of all compact sets in $X$, $F_\mathscr{P}$ will be $\beta X-X$ and hence $I_\mathscr{P}=\bigcap\limits_{p\in \beta X-X}M^p$ i.e., $cl_mC_K(X)=\bigcap\limits_{p\in \beta X-X}M^p$ [This last result is achieved independently in~\cite{Azarpanah2} [Proposition $5.6$]].
	\end{corollary}
	\begin{corollary}
		From Theorem $3.1$~\cite{Johnson}, $C_\psi(X)=\bigcap\limits_{p\in\beta X\setminus \upsilon X}M^p$ and so $C_\psi(X)$ is closed in the $m$-topology. Again by Theorem~\ref{Th4.14}(\ref{Th4.14(3)}), $cl_m(C_\psi(X))=\{f\in C(X):fg\in C^\psi_\infty(X)\text{ for each }g\in C(X)\}$. Thus $C_\psi(X)$ can also be written as $\{f\in C(X):fg\in C^\psi_\infty(X)\text{ for each }g\in C(X)\}$, this is an alternate formula for $C_\psi(X)$ [Compare with the known formula: $C_\psi(X)=\{f\in C(X):fg\in C^*(X)\text{ for each }g\in C(X)\}$, [Recorded in Section $3$,~\cite{Azarpanah2}]].
	\end{corollary}
	We conclude this section by establishing a characterization of pseudocompact spaces.
	\begin{theorem}
		The $U$-topology and the $m$-topology on $C(X)$ are equal if and only if the closures of $C_\mathscr{P}(X)$ in the respective topologies are equal for every choice of ideal $\mathscr{P}$ of closed sets in $X$. Therefore $X$ is pseudocompact if and only if for every choice of ideal $\mathscr{P}$ of closed sets in $X$, $C_\mathscr{P}(X)$ is dense in $C^\mathscr{P}_\infty(X)$ in the $m$-topology.
	\end{theorem}
	\begin{proof}
		If these two topologies are unequal, then $X$ is not pseudocompact and so there exists $f\in C^*(X)$ such that $Z(f)=\emptyset$ and $f^*(\beta X\setminus X)=\{0\}$. Consider $\mathscr{P}$, the ideal of bounded subsets of $X$. Then $C_\mathscr{P}(X)=C_\psi(X)$ and by Theorem~\ref{Th4.11}, $cl_UC_\mathscr{P}(X)=\{f\in C(X):f^*(\beta X\setminus\upsilon X)=\{0\}\}$ and $cl_mC_\mathscr{P}(X)=C_\psi(X)$. Clearly $f\in cl_UC_\mathscr{P}(X)\setminus cl_mC_\mathscr{P}(X)$, i.e., $cl_UC_\mathscr{P}(X)\neq cl_mC_\mathscr{P}(X)$.
	\end{proof}
	
\end{document}